\documentclass[11pt,letterpaper]{article}
\usepackage[T1]{fontenc}
\usepackage{lmodern}
\usepackage{amsmath,amssymb,amsthm,mathtools}
\usepackage[margin=1in]{geometry}
\usepackage{microtype}
\usepackage{tikz}
\usepackage{float}
\usepackage{xcolor}
\usepackage[colorlinks=true,linkcolor=blue!45!black,urlcolor=blue!45!black,citecolor=blue!45!black]{hyperref}
\hypersetup{
  pdftitle={Five-Point Hyperbolas on Power Curves and Near Lame-Oval Vertices},
  pdfauthor={George M. Georgiou},
  pdfkeywords={five-point conic, Lame oval, superellipse, power curve,
    equi-affine curvature, osculating conic, generalized Vandermonde determinant,
    Chebyshev system, Croft-Falconer-Guy Problem A33}
}

\newtheorem{theorem}{Theorem}
\newtheorem{proposition}[theorem]{Proposition}
\newtheorem{lemma}[theorem]{Lemma}
\newtheorem{corollary}[theorem]{Corollary}
\theoremstyle{remark}
\newtheorem{remark}[theorem]{Remark}
\newcommand{\R}{\mathbb{R}}

\title{Five-Point Hyperbolas on Power Curves and Near Lam\'e-Oval Vertices}
\author{\small George M. Georgiou\thanks{\small School of Computer Science and Engineering, California State University, San Bernardino}\\[3pt]
\small \href{mailto:georgiou@csusb.edu}{georgiou@csusb.edu}}
\date{}

\begin{document}
\maketitle

\begin{abstract}
Problem A33 of Croft--Falconer--Guy records Reznick's question about infinite
plane sets for which every five-point subset determines an ellipse or,
respectively, a hyperbola, and suggests that
\(|x|^{2.001}+|y|^{2.001}=1\) might yield only ellipses.  We give two results.
First, if \(p>2\) and \(a>0\), every five distinct points of the power curve
\(u=ay^p\), \(y>0\), determine a nondegenerate hyperbola; bounded subarcs
therefore give nonconic rectifiable examples for the hyperbolic side of
Reznick's question.  Second, every fixed one-sided five-point profile,
contracted toward an axial vertex of the Lam\'e oval
\(|x|^p+|y|^p=1\), eventually determines a nondegenerate hyperbola.
Consequently, the suggested Lam\'e oval with \(p=2.001\) has five-point subsets
that determine nondegenerate hyperbolas, so it does not yield only ellipses.
Symmetric profiles straddling the same vertex, however, determine ellipses.  We
also separate these results from the classical osculating-conic criterion
supplied by equi-affine curvature.
\end{abstract}

\medskip

\noindent{\small\emph{Key words and phrases.}
Five-point conic, Lam\'e oval, superellipse, power curve, equi-affine
curvature, osculating conic, generalized Vandermonde determinant, Chebyshev
system, Croft--Falconer--Guy Problem~A33.}

\paragraph{Relation to Problem A33.}
Problem A33 attributes the infinite-set question to Reznick and records the
Lam\'e-oval suggestion impersonally \cite[p.~43]{CFG}.  Theorem~\ref{thm:main}
and Corollary~\ref{cor:2001} give a direct negative answer to that suggestion:
the oval \(|x|^{2.001}+|y|^{2.001}=1\) has five-point subsets that determine
nondegenerate hyperbolas and therefore does not have the ``only ellipses''
property proposed in A33.  Classical
osculating-conic theory already gives a short route to the existence of
hyperbolic five-point subsets near points of negative equi-affine curvature.
Thorbergsson and Umehara noted that they were unaware of results on
sextactic points for curves outside their strictly convex class
\cite[Introduction]{TU}; the flat-vertex geometry considered here belongs to
that broader regime.  This last observation is included only as historical
context concerning work outside their strictly convex setting; it is not a
general priority claim.
The contributions below go in different directions: Proposition~\ref{prop:model}
classifies \emph{every} five-point subset of a power arc, while
Theorem~\ref{thm:main} controls each prescribed one-sided profile as it is
contracted toward a flat Lam\'e-oval vertex.

\section{Statements}

Unless stated otherwise, \(p>2\) is real.  Put
\[
 S_p=\{(x,y)\in\R^2:|x|^p+|y|^p=1\},
 \qquad
 \Gamma_{p,a}=\{(ay^p,y):y>0\}\quad(a>0).
\]
We work near the vertex \((1,0)\) in the coordinates
\begin{equation}\label{eq:coords}
 u=1-x,\qquad y=y,
\end{equation}
which is an affine change and therefore preserves the type of a conic.  We
write a conic as
\begin{equation}\label{eq:conic}
 F+Ey+Cy^2+Du+B\,uy+A\,u^2=0,
\end{equation}
whose quadratic part has discriminant \(\Delta=B^2-4AC\): the conic is of
hyperbolic type when \(\Delta>0\) and of elliptic type when \(\Delta<0\).

We record once and for all the following elementary fact, which supplies both
uniqueness and nondegeneracy throughout.

\begin{lemma}\label{lem:convex}
For \(p>1\), the Lam\'e oval \(S_p\) is the boundary of a strictly convex
body, while \(\Gamma_{p,a}\) is a strictly convex graph for every \(a>0\).
Consequently no three distinct points of either curve are collinear, and any
five distinct points of either curve lie on a unique nondegenerate conic.
\end{lemma}

\begin{proof}
For \(p>1\) the function \(t\mapsto|t|^p\) is strictly convex, hence so is
\((x,y)\mapsto|x|^p+|y|^p\); its sublevel set \(\{|x|^p+|y|^p\le1\}\) is
therefore a strictly convex body, and the boundary of such a body contains no
line segment.  Concretely, on \(0<y<1\) the graph \(x=(1-y^p)^{1/p}\) satisfies
\[
 \frac{d^2x}{dy^2}=-(p-1)y^{p-2}(1-y^p)^{1/p-2}<0,
\]
and \(u=ay^p\) satisfies \(u''=ap(p-1)y^{p-2}>0\) on \(y>0\).  A line therefore
meets either curve in at most two points, so no three of the points are
collinear.  The standard five-point theorem now gives a unique nondegenerate
conic \cite[\S2B]{TU}.  For completeness, existence follows from the five
homogeneous linear incidence conditions on the six conic coefficients.  If
two distinct conics passed through the points, B\'ezout's theorem would allow
at most four intersections unless they shared a line component \(L\).  In the
latter case at most two of the five points could lie on \(L\), and every common
point off \(L\) would have to be the single intersection of the two residual
lines.  Thus at most \(2+1=3\) of the five points could be common, again a
contradiction.  Finally, a degenerate real conic containing five distinct real
points cannot be of definite rank two, since its real locus would consist of a
single point.  Otherwise it is a union of two real lines, possibly coincident,
forcing three of the five points onto one line.  Hence the conic is
nondegenerate.
\end{proof}

\begin{theorem}\label{thm:main}
Fix a real number \(p>2\) and fixed numbers \(0<t_1<t_2<t_3<t_4<t_5\).  For
\(0<\varepsilon<t_5^{-1}\) put
\[
 P_i(\varepsilon)
 =\left(\bigl(1-(t_i\varepsilon)^p\bigr)^{1/p},\,t_i\varepsilon\right)
 \in S_p .
\]
There is an \(\varepsilon_0=\varepsilon_0(p,t_1,\dots,t_5)>0\) such that,
whenever
\(0<\varepsilon<\varepsilon_0\), the unique conic through
\(P_1(\varepsilon),\dots,P_5(\varepsilon)\) is a nondegenerate hyperbola.
\end{theorem}

\begin{corollary}\label{cor:2001}
The curve \(|x|^{2.001}+|y|^{2.001}=1\) has five-point subsets whose conic is a
hyperbola.  Thus the specific ``only ellipses'' suggestion in Problem A33 is
false.
\end{corollary}

\section{Classical osculating-conic context}\label{sec:affine}

For a \(C^4\) graph \(u=f(y)\) with \(f''>0\), the equi-affine curvature,
with the usual normalization, is
\begin{equation}\label{eq:affine-curvature}
 \kappa_{\mathrm{aff}}
 =\frac{3f''f''''-5(f''')^2}{9(f'')^{8/3}}.
\end{equation}
This is Blaschke's formula \cite[p.~14, formula~(83)]{Blaschke}.  Classical
affine differential geometry identifies the unique conic having five-point
contact with the curve as its osculating conic; its type is elliptic, parabolic
or hyperbolic according as \(\kappa_{\mathrm{aff}}\) is positive, zero or
negative.  Thorbergsson and Umehara describe in their Introduction the
convergence of conics through five distinct points to the osculating conic as
all five points coalesce; their \S2C gives the equivalent multiplicity-five
definition, and they recall the curvature/type criterion before
Corollary~4.2 \cite[Introduction; \S2C; and the discussion preceding
Corollary~4.2]{TU}.

For the power graph \(f(y)=ay^p\), direct substitution in
\eqref{eq:affine-curvature} gives
\begin{equation}\label{eq:model-affine-numerator}
 3f''f''''-5(f''')^2
 =-a^2p^2(p-1)^2(p-2)(2p-1)y^{2p-6}.
\end{equation}
Thus its osculating conic is hyperbolic exactly when \(p>2\).  This local fact
is consistent with Proposition~\ref{prop:model}, but is much weaker than that
proposition, which treats arbitrary five distinct points rather than a
coalescing configuration.

For the actual Lam\'e branch, put
\[
 f_p(y)=1-(1-y^p)^{1/p}\qquad(0<y<1).
\]
Here
\[
 f_p''(y)=(p-1)y^{p-2}(1-y^p)^{1/p-2}>0,
\]
and a calculation gives
\begin{align}
 3f_p''f_p''''-5(f_p''')^2
 &=(p-1)^2(2p-1)y^{2p-6}(1-y^p)^{2/p-6}
 \notag\\[-2pt]
 &\quad\times\bigl((p+1)y^p(1-y^p)-(p-2)\bigr).
 \label{eq:lame-affine-numerator}
\end{align}
In particular,
\begin{equation}\label{eq:lame-affine-asymptotic}
 \kappa_{\mathrm{aff}}(y)
 \sim-\frac{(2p-1)(p-2)}{9(p-1)^{2/3}}
 y^{-2(p+1)/3}\longrightarrow-\infty
 \qquad(y\downarrow0).
\end{equation}
The analogous divergence at a smooth finite-type inflection endpoint is
classical \cite[Remark~(iii) following Proposition~5.1]{TU}.  That reference
works in a smooth finite-type setting, where the contact order is an integer.
For noninteger \(p\), the Lam\'e oval has only finite differentiability at the
vertex, so the cited statement does not apply there verbatim;
\eqref{eq:lame-affine-asymptotic} follows directly from
\eqref{eq:lame-affine-numerator} for every real \(p>2\).

Thus the osculating conic at every sufficiently small regular point \(y>0\) is
a hyperbola, and five distinct points sufficiently close to that point also
determine a hyperbola.  This is the classical route to
Corollary~\ref{cor:2001}; it does not imply Theorem~\ref{thm:main}, where the ratios
\(t_1:\cdots:t_5\) remain fixed while the entire profile approaches the flat
vertex.  The closed form \eqref{eq:lame-affine-numerator}, valid on the whole
open branch, yields the additional global sign information below.

Writing \(z=y^p\), the sign in \eqref{eq:lame-affine-numerator} is the sign of
\((p+1)z(1-z)-(p-2)\).  Since \(z(1-z)\le1/4\), the affine curvature on the
open quarter-arc is strictly negative for \(p>3\); for \(p=3\) it is negative
except at \(z=1/2\), the arc midpoint \(x=y=2^{-1/3}\), where it vanishes; and
for \(2<p<3\) it takes both signs.

The exact model result and this sign analysis suggest a natural global
question: for every \(p\ge3\), does every five-point subset of the open
first-quadrant arc of \(S_p\) determine a hyperbola?  The positive-minor
orientation of Lemma~\ref{lem:vander} would suffice, and the proof of
Theorem~\ref{thm:main} obtains that orientation for each fixed one-sided profile
after sufficiently strong contraction.  It cannot, however, hold on the whole
arc.

Indeed, let \(\sigma(x,y)=(y,x)\), which in the coordinates \eqref{eq:coords}
is the affine involution \(\sigma(u,y)=(1-y,1-u)\).  Take the fixed midpoint of
the quarter-arc together with two distinct
nonfixed points and their \(\sigma\)-images, the two orbits also being distinct,
and let \(Q=0\) be their unique conic.  Since the five-point set is
\(\sigma\)-invariant, uniqueness gives \(Q\circ\sigma=\lambda Q\); applying
\(\sigma\) twice shows that \(\lambda=\pm1\).  The negative sign would make
\(Q\) vanish on the fixed line \(u+y=1\), so this line would be a component of
\(Q\).  That line is the diagonal \(y=x\), which meets the arc only at the
midpoint, so the residual line would contain the other four points, contrary to
Lemma~\ref{lem:convex}.  Hence \(Q\circ\sigma=Q\).  Comparing the coefficients
of \(u^2\) and \(y^2\) in \eqref{eq:conic} gives
\[
  A=C.
\]
Thus the strict alternating sign pattern behind Proposition~\ref{prop:model}
is impossible for these symmetric configurations, and full-arc total
positivity is not a viable route.  If \(A\ne0\), normalization by \(A=1\) gives
\(C=1\) and
\[
  \Delta=B^2-4,
\]
so the conic is hyperbolic exactly when \(\lvert B\rvert>2\).  If \(A=0\), then
also \(C=0\), and nondegeneracy forces \(B\ne0\), so the conic is already
hyperbolic.  At \(p=3\) the symmetric confluent limit at the midpoint is the
parabolic osculating conic; after normalization by \(A=1\), it has
\(\lvert B\rvert=2\).  Thus \(\lvert B\rvert=2\) is the sharp confluent boundary
for the symmetric subfamily, while the local osculating-conic criterion alone
does not answer the global question.

There is also a classical obstruction within the smooth strictly convex class
of Thorbergsson and Umehara, where a closed curve is called strictly convex if
it is convex and has no inflection points \cite[\S2C]{TU}.
Mukhopadhyaya's theorem, in their modern treatment, gives such a curve at least
three inscribed osculating ellipses
\cite[Theorem~4.1]{TU}.  Such a curve therefore has elliptic five-point
configurations nearby.  Smooth Lam\'e ovals with flat vertices have inflection
points in this projective-differential sense; for noninteger \(p\), limited
differentiability at the vertices also places the oval outside their smoothness
hypotheses.  In either case, it still bounds a strictly convex body in the
ordinary convex-geometric sense of Lemma~\ref{lem:convex}.

\section{The sign lemma}

The following standard generalized-Vandermonde sign fact is a special case of
the theory of extended complete Chebyshev systems
\cite[Chapter~I]{KarlinStudden}; the proof is included for completeness.

\begin{lemma}[Generalized Vandermonde signs]\label{lem:vander}
If \(0<t_1<\cdots<t_m\) and \(\lambda_1<\cdots<\lambda_m\) are real, then
\[
 \det\bigl(t_i^{\lambda_j}\bigr)_{i,j=1}^m>0.
\]
Consequently, if \(\lambda_1<\cdots<\lambda_{m+1}\) and
\(\gamma_1,\dots,\gamma_{m+1}>0\), then the \(m\times(m+1)\) matrix
\(\bigl(\gamma_j t_i^{\lambda_j}\bigr)\) has rank \(m\), and every nonzero
vector in its nullspace has strictly alternating signs.
\end{lemma}

\begin{proof}
For any increasing subfamily \(\lambda_{j_1}<\cdots<\lambda_{j_k}\), the
Wronskian of \(t^{\lambda_{j_1}},\dots,t^{\lambda_{j_k}}\) is
\[
 t^{\sum_{r}\lambda_{j_r}-k(k-1)/2}
 \prod_{r<s}(\lambda_{j_s}-\lambda_{j_r})>0
 \qquad(t>0),
\]
as one sees by factoring \(t^{\lambda_{j_r}}\) from each column and
\(t^{-(i-1)}\) from each row, the remaining determinant being a Vandermonde
determinant in the \(\lambda\)'s built from the monic falling factorials.  All
Wronskians of initial subfamilies being positive, the family is an extended
complete Chebyshev system; by repeated use of Rolle's theorem no nontrivial
linear combination has \(k\) positive zeros, so the evaluation determinant at
distinct increasing points never vanishes.  Its sign is constant on the
connected set \(\{0<t_1<\cdots<t_m\}\) and is found by letting the points
coalesce, where the determinant is asymptotic to a positive multiple of the
Wronskian; this gives the displayed positive sign.

For the second statement, each maximal minor of \(\bigl(\gamma_jt_i^{\lambda_j}\bigr)\)
is a determinant of the first kind multiplied by the product of the \(\gamma\)'s
of the retained columns, hence positive; in particular the rank is \(m\).  The
vector \(v\) with \(v_j=(-1)^j\det M_{\hat\jmath}\), where \(M_{\hat\jmath}\)
omits column \(j\), spans the nullspace, and its signs alternate.
\end{proof}

\section{The model curve}\label{sec:model}

Everything is driven by the pure power curve
\(\Gamma_{p,a}\), for which the conclusion is exact and requires no smallness
whatsoever.

\begin{proposition}\label{prop:model}
Let \(p>2\) and \(a>0\).  Then any five distinct points of
\(\Gamma_{p,a}\) lie on a unique conic, and that conic is a nondegenerate
hyperbola.  Moreover, in the normalization \(A=1\) of \eqref{eq:conic} the
coefficient vector
\((F,E,C,D,B,A)\) has signs \(-,+,-,+,-,+\); in particular \(C<0\).
\end{proposition}

\begin{proof}
At the point of \(\Gamma_{p,a}\) with height \(y\), the six monomials of
\eqref{eq:conic} take the values
\[
 1,\quad y,\quad y^2,\quad a y^{p},\quad a y^{p+1},\quad a^2y^{2p},
\]
so the \(5\times6\) evaluation matrix at five distinct heights
\(0<y_1<\cdots<y_5\) is exactly of the shape treated in
Lemma~\ref{lem:vander}, with exponents
\[
 0<1<2<p<p+1<2p
\]
(the ordering uses \(p>2\) for \(2<p\) and \(p>1\) for \(p+1<2p\)) and positive
factors \(1,1,1,a,a,a^2\).  Hence the matrix has rank five, so the conic is
unique, and its coefficient vector alternates in sign.  Normalizing the last
coordinate \(A\) to be \(+1\) yields the sign pattern \(-,+,-,+,-,+\), so
\(C<0\) and
\[
 \Delta=B^2-4AC=B^2-4C>0 .
\]
The conic is therefore a hyperbola, and its nondegeneracy is already supplied
by Lemma~\ref{lem:convex}.
\end{proof}

\begin{corollary}\label{cor:powerarc}
If \(p>2\), \(a>0\), and \(0<r<s<\infty\), then the compact arc
\[
 \Gamma_{p,a}[r,s]=\{(ay^p,y):r\le y\le s\}
\]
is a nonconic real-analytic rectifiable arc every five distinct points of
which determine a nondegenerate hyperbola.
\end{corollary}

\begin{proof}
The arc is real analytic and rectifiable, and the five-point assertion is
Proposition~\ref{prop:model}.  If a conic contained the arc, then substituting
\(u=ay^p\) in its equation would give a linear combination of the six
distinct powers
\[
 1,\quad y,\quad y^2,\quad y^p,\quad y^{p+1},\quad y^{2p}
\]
that vanished on \([r,s]\).  Evaluating at any six distinct points of the
interval contradicts Lemma~\ref{lem:vander}.
\end{proof}

\begin{remark}[Anisotropic scaling]\label{rem:scaling}
For \(\lambda>0\) the map \(L_\lambda(u,y)=(\lambda^pu,\lambda y)\) is linear
and invertible, hence carries conics to conics of the same type, and it
preserves \(\Gamma_{p,a}\).  Proposition~\ref{prop:model} is therefore invariant
under contracting a configuration toward the origin, which is precisely why
the column scalings used in the next proof are the natural ones.
\end{remark}

\section{Proof of Theorem~\ref{thm:main}}

Along the upper right branch of \(S_p\),
\begin{equation}\label{eq:expansion}
 u=1-(1-y^p)^{1/p}=\alpha y^p+O(y^{2p}),
 \qquad \alpha=\frac1p,\quad y\downarrow0,
\end{equation}
the expansion proceeding in powers of \(y^p\) alone.  Write \(y_i=\varepsilon
t_i\) and consider the \(5\times6\) evaluation matrix of \eqref{eq:conic} with
columns ordered as \(1,\ y,\ y^2,\ u,\ uy,\ u^2\).  After dividing these columns
respectively by
\[
 1,\ \varepsilon,\ \varepsilon^2,\ \varepsilon^p,
 \ \varepsilon^{p+1},\ \varepsilon^{2p},
\]
equation \eqref{eq:expansion} shows that the matrix converges, as
\(\varepsilon\downarrow0\), to the matrix with columns
\begin{equation}\label{eq:limitcols}
 1,\quad t,\quad t^2,\quad \alpha t^p,\quad \alpha t^{p+1},
 \quad \alpha^2t^{2p}
\end{equation}
evaluated at \(t=t_1,\dots,t_5\); that is, to the evaluation matrix of the
model curve \(\Gamma_{p,\alpha}\) at heights \(t_1,\dots,t_5\)
(Remark~\ref{rem:scaling} explains the choice of divisors).  Let
\[
 (f,e,c,d,b,1)
\]
be its nullvector normalized so that the last coordinate is one.  By
Proposition~\ref{prop:model},
\begin{equation}\label{eq:cnegative}
 c<0 .
\end{equation}

The maximal minor obtained by omitting the \(u^2\)-column is positive in the
limit, hence nonzero for all sufficiently small \(\varepsilon\); therefore
\(A\neq0\) and we may normalize \eqref{eq:conic} by \(A=1\).  The normalized
nullvector is a ratio of minors and so depends continuously on
\(\varepsilon\).  Undoing the column divisions gives
\begin{align*}
 F&=\varepsilon^{2p}(f+o(1)),
 &E&=\varepsilon^{2p-1}(e+o(1)),
 &C&=\varepsilon^{2p-2}(c+o(1)),\\
 D&=\varepsilon^p(d+o(1)),
 &B&=\varepsilon^{p-1}(b+o(1)).
\end{align*}
Note that \(B^2\) and \(C\) both contribute at the same order
\(\varepsilon^{2p-2}\); no term dominates the other, and it is the sign
\eqref{eq:cnegative} that is decisive.  Indeed
\[
 \frac{\Delta_\varepsilon}{\varepsilon^{2p-2}}
 =\frac{B^2-4C}{\varepsilon^{2p-2}}
 \longrightarrow b^2-4c>0 ,
\]
so \(\Delta_\varepsilon>0\) for all sufficiently small \(\varepsilon>0\) and the
conic is of hyperbolic type.  Its uniqueness and nondegeneracy follow from
Lemma~\ref{lem:convex}.  Finally the affine change \eqref{eq:coords} does not
alter the conic type. \qed

\section{One-sidedness is essential}\label{sec:straddle}

The hypothesis that the five points lie on one side of the vertex cannot be
dropped, and the reason is the same condition \(p>2\), here in the equivalent
form \(2^p>4\), that drives Theorem~\ref{thm:main}.

\begin{proposition}\label{prop:straddle}
Let \(p>1\), \(p\neq2\), and for small \(\varepsilon>0\) let
\[
 u_j=1-\bigl(1-(j\varepsilon)^p\bigr)^{1/p}\qquad(j=1,2),
\]
so that the five points of \(S_p\) with \(x>0\) and
\(y=0,\pm\varepsilon,\pm2\varepsilon\) are \((0,0)\), \((u_1,\pm\varepsilon)\),
\((u_2,\pm2\varepsilon)\) in the coordinates \eqref{eq:coords}.  For all
sufficiently small \(\varepsilon>0\) the unique conic through these five points
is
\begin{equation}\label{eq:straddleconic}
 (4u_1-u_2)\,u^2
 +\frac{u_1u_2(u_1-u_2)}{\varepsilon^{2}}\,y^2
 +(u_2^2-4u_1^2)\,u=0,
\end{equation}
and it is a nondegenerate ellipse when \(p>2\) and a nondegenerate hyperbola
when \(1<p<2\).
\end{proposition}

\begin{proof}
Direct substitution verifies that \eqref{eq:straddleconic} vanishes at
\((0,0)\), at \((u_1,\pm\varepsilon)\) and at \((u_2,\pm2\varepsilon)\); for
instance at \((u_1,\pm\varepsilon)\) the left-hand side equals
\[
 u_1\bigl[(4u_1-u_2)u_1+u_2(u_1-u_2)+u_2^2-4u_1^2\bigr]=0 .
\]
By \eqref{eq:expansion},
\(u_j=\alpha(j\varepsilon)^p\bigl(1+O(\varepsilon^p)\bigr)\), so
\begin{equation}\label{eq:key}
 4u_1-u_2=\alpha\varepsilon^p
 \bigl(4-2^p+O(\varepsilon^p)\bigr),
\end{equation}
which is nonzero for small \(\varepsilon\) because \(p\neq2\); hence
\eqref{eq:straddleconic} really is a conic, and by Lemma~\ref{lem:convex} it is
the unique one through the five points.  Its quadratic part has \(B=0\), so
\[
 \Delta=-4(4u_1-u_2)\cdot\frac{u_1u_2(u_1-u_2)}{\varepsilon^{2}} .
\]
Here \(u_1u_2>0\) and \(u_1-u_2<0\).  By \eqref{eq:key} the factor
\(4u_1-u_2\) is negative for \(p>2\) and positive for \(1<p<2\), whence
\(\Delta<0\) in the first case and \(\Delta>0\) in the second.  Nondegeneracy
follows from Lemma~\ref{lem:convex}.
\end{proof}

Thus, for every \(p>2\), arbitrarily close to the same vertex there are
one-sided configurations determining hyperbolas and symmetric straddling
configurations determining ellipses.  This is why the abstract speaks of
contracting a \emph{fixed one-sided profile}, and it also shows that the answer
to A33 for \(S_p\) is genuinely mixed rather than uniformly hyperbolic.

Figure~\ref{fig:profiles} displays this contrast on the anisotropically
magnified model when \(p=3\).

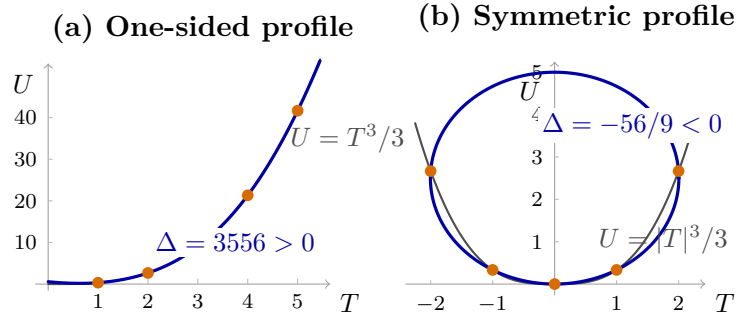
\begin{figure}[H]
\centering
\begin{tikzpicture}[font=\small, line cap=round, line join=round]
  \begin{scope}[x=0.66cm,y=0.055cm]
    \draw[->,gray!75] (-0.25,0) -- (5.65,0)
      node[below right,black] {$T$};
    \draw[->,gray!75] (0,-1.5) -- (0,53);
    \node[black,anchor=east] at (-0.08,48) {$U$};
    \foreach \xx in {1,...,5}
      \draw[gray!55] (\xx,0) -- ++(0,1.2)
        node[below=2pt,black] {\scriptsize \(\xx\)};
    \foreach \yy in {10,20,30,40}
      \draw[gray!55] (0,\yy) -- ++(0.08,0)
        node[left=2pt,black] {\scriptsize \(\yy\)};
    \draw[black!70,thick,domain=0:5.45,samples=150,variable=\t]
      plot ({\t},{\t^3/3});
    \draw[blue!65!black,very thick,domain=0:5.45,samples=180,variable=\t]
      plot
      ({\t},
       {(46.6666667*\t-350
       +351.1409973*sqrt(1-0.279318735*\t
       +0.0288402271*\t^2))/2});
    \foreach \tt in {1,...,5}
      \fill[orange!85!black] (\tt,{\tt^3/3}) circle[radius=2.2pt];
    \node[anchor=south,font=\bfseries] at (3.1,55.5)
      {(a) One-sided profile};
    \node[blue!65!black,fill=white,inner sep=1.5pt] at (3.75,10)
      {\(\Delta=3556>0\)};
    \node[black!70,anchor=west] at (4.65,35)
      {\(U=T^3/3\)};
  \end{scope}

  \begin{scope}[xshift=6.7cm,x=0.82cm,y=0.56cm]
    \draw[->,gray!75] (-2.4,0) -- (2.43,0)
      node[below right,black] {$T$};
    \draw[->,gray!75] (0,-0.25) -- (0,5.25);
    \node[black,anchor=east] at (-0.06,4.55) {$U$};
    \foreach \xx in {-2,-1,1,2}
      \draw[gray!55] (\xx,0) -- ++(0,0.08)
        node[below=2pt,black] {\scriptsize \(\xx\)};
    \foreach \yy in {1,2,3,4,5}
      \draw[gray!55] (0,\yy) -- ++(0.06,0)
        node[left=2pt,black] {\scriptsize \(\yy\)};
    \draw[black!70,thick,domain=-2.25:2.25,samples=150,variable=\t]
      plot ({\t},{abs(\t)^3/3});
    \draw[blue!65!black,very thick,domain=0:360,samples=180,variable=\ang]
      plot ({2.0044593*cos(\ang)},{2.5+2.5*sin(\ang)});
    \foreach \tt in {-2,-1,0,1,2}
      \fill[orange!85!black] (\tt,{abs(\tt)^3/3}) circle[radius=2.2pt];
    \node[anchor=south,font=\bfseries] at (0.35,5.7)
      {(b) Symmetric profile};
    \node[blue!65!black,fill=white,inner sep=1.5pt] at (1.25,3.75)
      {\(\Delta=-56/9<0\)};
    \node[black!70,anchor=west] at (0.55,1.05)
      {\(U=|T|^3/3\)};
  \end{scope}
\end{tikzpicture}
\small
\caption{Anisotropic blow-up \(T=y/\varepsilon\), \(U=u/\varepsilon^p\) for
\(p=3\).  Black curves are the limiting Lam\'e models and the dots are the
defining points.  Both conic equations are normalized so that the coefficient
of \(U^2\) is one.  In (a), the coefficient vector
\((F,E,C,D,B,A)=
(-200,\frac{1330}{3},-\frac{3101}{9},350,-\frac{140}{3},1)\)
displays the alternating sign pattern of Proposition~\ref{prop:model} and has
\(\Delta=3556\).  The symmetric points in (b) lie on the ellipse
\(U^2+\frac{14}{9}T^2-5U=0\), for which \(\Delta=-\frac{56}{9}\).}
\label{fig:profiles}
\end{figure}

\paragraph{A concrete numerical illustration.}
Choose \(p=2.001\), \(t_i=i\), and \(\varepsilon=10^{-3}\).  With \(A=1\), a
120-digit computation gives, to the displayed precision,
\[
\begin{split}
 u^2&-0.0118470556945\,uy-0.0241022648416\,y^2\\
 &+0.0485159650138\,u+4.05547453\cdot10^{-8}\,y
 -1.16146502\cdot10^{-11}=0,
\end{split}
\]
with
\[
 \Delta=0.0965494120951>0,
\]
so the computed coefficients indicate a hyperbola.  These rounded values are
only a numerical illustration; Corollary~\ref{cor:2001} follows rigorously
from Theorem~\ref{thm:main}.

\section{Why the threshold \texorpdfstring{\(p=2\)}{p=2} appears, and what
happens below it}

At \(p=2\) the exponents \(2\) and \(p\) coincide, so the strict
generalized-Vandermonde ordering in \eqref{eq:limitcols} fails and
Lemma~\ref{lem:vander} no longer applies; consistently, \(S_2\) is the circle
itself, so every five-point conic is that ellipse.  For every \(p>2\), however
small \(p-2\) is, the exponents become strictly ordered and the
alternating-minor argument forces the hyperbolic sign near the vertex.

For \(1<p<2\) the exponents are ordered instead as
\(0<1<p<2<p+1<2p\), so the alternating pattern places the \(y^2\)-coefficient
in the fourth slot and gives \(c>0\).  The sign of \(b^2-4c\) is then
\emph{not} determined by Lemma~\ref{lem:vander} alone.  In fact both types
occur already on the model curve.  For \(p=3/2\), \(a=2/3\), and normalization
\(A=1\), the five heights \(1,4,9,16,25\) give
\[
 B=-\frac{1330}{137},\qquad C=\frac{42980}{1233},\qquad
 \Delta=-\frac{7632940}{168921}<0,
\]
whereas the heights \(1,4,9,36,625\) give
\[
 B=-\frac{3724}{153},\qquad C=\frac{67088}{459},\qquad
 \Delta=\frac{182224}{23409}>0.
\]
Both conics are nondegenerate by Lemma~\ref{lem:convex}.  The same perturbation argument used in
Theorem~\ref{thm:main} shows that the corresponding fixed
one-sided profiles on the Lam\'e oval retain these respective types when
contracted sufficiently close to the vertex.  Thus the direct analogue of
Theorem~\ref{thm:main} below \(p=2\) is false, although
Proposition~\ref{prop:straddle} does reverse cleanly there.

\end{document}